\documentclass[12pt, a4paper, reqno]{article}
\usepackage{geometry}
\date{February 23, 2023}
\usepackage[shortlabels]{enumitem}
\usepackage{graphicx}
\usepackage[utf8]{inputenc}
\usepackage[T1]{fontenc}
\usepackage[T2A]{fontenc}
\usepackage[english]{babel}

\usepackage{tikz-cd} 
\usepackage{xfrac}  
\usepackage{calrsfs}   
\usepackage{mathrsfs}   
\usepackage{amssymb}
\usepackage{amsthm}
\usepackage{amssymb}
\newtheorem{theorem}{Theorem}[section]
\newtheorem{lemma}[theorem]{Lemma}
\theoremstyle{definition}
\newtheorem{definition}{Definition}[section]
\newtheorem*{acknowledgment}{Acknowledgment}
\newtheorem{proposition}{Proposition}[section]
\newtheorem{statement}{Statement}
\theoremstyle{remark}
\newtheorem*{remark}{Remark}
\usepackage{csquotes}

\title{A colored Tverberg type theorem for unavoidable complexes}
\author{Mikhail Bludov}
\usepackage[hidelinks]{hyperref}

\begin{document}
\sloppy
\maketitle

\begin{abstract}
The main result of this paper is a ``colored Tverberg theorem for rainbow-unavoidable complexes''. This theorem may be considered as a merging of two theorems: ``Tverberg theorem for collectively unavoidable complexes'' and ``balanced colored Tverberg theorem''. The main tool for the proof is discrete Morse theory.
\end{abstract}

\section{Introduction}

Tverberg's theorem \cite{T} is a statement about certain partitions of finite set of points in $\mathbb{R}^d$. There are many ways to extend this theorem using different concepts and approaches. Such family of extensions is known as ``Tverberg type theorems'' and appears to be one of the main research theme in topological combinatorics. Among the most important results of the last decade in this field we distinguish the counterexamples to the general ``topological Tverberg theorem'' \cite{CTTC, ETP, BPSCTT}.

Among other results, ``Balanced Van Kampen-Flores''  \cite[Theorem 1.2]{SCCTT} was acquired for balanced complexes  (definition \ref{balanced}). Then in  \cite{TverType} this result was extended to ``Tverberg theorem for collectively unavoidable complexes'' (Theorem \ref{TCUC}). The approach to this theorem relies on methods of discrete Morse theory and on the concepts of ``balanced'' and ``collectively unavoidable'' $r$-tuples of simplicial complexes (definition \ref{ColUnComp}). In this paper using the concept of $(r,s)$-unavoidability we acquire a slight extension of Theorem \ref{TCUC}.

``Colored Tverberg type theorems'' \cite{ZivCol, BaranyLarman} is another research direction. In \cite{ColTver}, among other things, was obtained the ``balanced colored Tverberg theorem'' \ref{BCT}. This theorem requires the concept of ``balanced'' cimplicial complex and discrete Morse theory and is relative to Theorem \ref{TCUC}. The main result of this paper (Theorem \ref{CTCRUC}) is an extension of Theorem \ref{BCT} with the concept of ``collectively unavoidable'' $(r,s)$-tuples. We borrow the idea for the proof from the initial theorem \ref{BCT} but we add some important details.

\subsection{Preliminaries}

 \begin{definition}
Assume that $[m]$ is the set $\{1,\dots.m\}$ and $2^{[m]}$ is the set of all subsets of $[m]$. Let $\Delta_{[m]}$ be a combinatorial simplex on $m$ vertices of dimension $N=m-1$. We identify a simplex $\Delta^{N}$ with $\Delta_{[m]}$ and with $2^{[m]}$.
 
\end{definition}

 \begin{definition}[Definition 1.5.1 in \cite{UBUT}]
An abstract simplicial complex is a pair $(V,K)$, where $V$ is a set and $ K \subseteq 2^V$ is a hereditary system of subsets of $V$; that is, we require that $F \in K$ and $G \subseteq F$ imply $G \in K$ (in particular, $\emptyset \in K$ whenever $K \neq \emptyset$). The sets in $K$ are called (abstract) simplices. Further, we define the dimension ${dim(K)\colon= \max\{|F|-1 : F \in K\}}$.
 \end{definition}

 \begin{definition}[Definition 4.2.1 in \cite{UBUT}]
Let $K$ and $L$ be simplicial complexes. The join $K * L$ is the simplicial complex with vertex set $V(K) \uplus V(L)$ and with the set of simplices $\{F\uplus G: F \in K, G \in L\}$. In words, to construct the join, we first take a disjoint union of the vertex sets, and then we combine every simplex of K with every simplex of L.
 \end{definition}

 \begin{definition}[Section 6 in \cite{UBUT}]
 Let $\mathcal{K} = \langle K_{1}, \dots, K_{r} \rangle$  be a family of simplicial complexes, $K_i \subset 2^{[m]}$. Then the \textit{deleted join} of this family is given by $$(\mathcal{K}^{*}_{\Delta})=K_{1}*_{\Delta} \dots *_{\Delta} K_{r}=\{\uplus_{i=1}^{r} A_{i}: A_i \in K_j, ~ \forall i,j ~ A_i \cap A_j = \emptyset \} \subset (2^{[m]})^{*r}.$$
 \end{definition}

  \begin{definition}[\cite{TverType}]

 The \textit{symmetrized deleted join} of $\mathcal{K}$ is defined as  $$SymmDelJoin(\mathcal{K})=\bigcup_{\pi \in S_{r}}K_{\pi_{1}}*_{\Delta} \dots *_{\Delta} K_{\pi_{r}} \subseteq (2^{[m]})^{*r},$$ where $S_r$ is the symmetric group. An element $A_{1} \sqcup \dots \sqcup A_{r} \in SymmDelJoin(\mathcal{K})$ will be written as $(A_{1}, \dots, A_{r}, B)$, where $B$ is $[m] \setminus \bigcup_{i=1}^{r} A_{i}$. In other words, $(A_{1}, \dots, A_{r}, B)$ is a partition of $[m]$.
 \end{definition}

 \begin{definition} 
A \textit{coloring} of vertices $[m]$ by $k+1$ color is a partition $[m]=C_{1} \uplus \dots \uplus C_{k+1}$ into “monochromatic” subsets $C_i$. A vertex subset $V$ is called a rainbow set or a rainbow simplex if and only if $|V \cap C_{i}| \leq 1$ for each $i=1, \dots, k+1$. Simplicial complex $K$ is called rainbow if all its simplices are rainbow. $Col\Delta_{[m]}$ is the simplicial complex of all rainbow simplices from $\Delta_{[m]}$.
  \end{definition}

 \begin{definition} 
Let $K\subseteq \Delta_{[m]}$ be a simplicial complex. The \textit{$d$-skeleton $K^{(d)}$} of $K$ is a subcomplex of all simplices from $K$ of dimension at most $d$. 
 \end{definition}

\begin{definition} \label{balanced} We say that a simplicial complex $K \subset 2^{m}$ is  \textit{$(m,k)$-balanced} if $$\Delta_{[m]}^{(k-1)} \subseteq K \subseteq \Delta_{[m]}^{(k)}.$$ Suppose $[m]$ is colored, then $K$ is called  \textit{$(m,k)$-rainbow balanced} if $$Col\Delta_{[m]}^{(k-1)} \subseteq K \subseteq Col\Delta_{[m]}^{(k)}.$$
 \end{definition}

 \subsection{Collectively unavoidable complexes}

Collectively unavoidable complexes as self-interest object appeared in \cite{CUC}. In \cite{TUC} the authors study topological properties of $r$-unavoidable complexes. Here we will introduce the definition. 

\begin{definition}
A simplicial complex $K \subseteq 2^{[m]}$ is called collectively $r$-unavoidable if for each ordered collection $(A_1, \dots, A_r)$ of disjoint sets in $[m]$ there exists $i$ such that $A_i \in K$
\end{definition}

In \cite{TUC} there is a natural extension: 
\begin{definition}
A simplicial complex $K \subseteq 2^{[m]}$ is called collectively $(r,s)$-unavoidable if for each ordered collection $(A_1, \dots, A_r)$ of disjoint sets in $[m]$ there exists a set of indices $\{i_1, \dots, i_s\}$ such that $A_{i_j} \in K$ for all $j=1, \dots, s$.
\end{definition}

We are interested in collectively $r$-unavoidable $r$-tuples of complexes. They were originally studied in \cite{ABC}. Here is the definition:
\begin{definition}
\label{ColUnComp}
An ordered $r$-tuple $K = \langle K_1, \dots, K_{r}\rangle$ of subcomplexes of $2^{[m]}$ is collectively $r$-unavoidable if for each ordered collection $(A_1, \dots, A_r)$ of disjoint sets in $[m]$ there exists $i$ such that $A_i \in K_i$.
\end{definition}

\begin{definition} 
An ordered $r$-tuple $K = \langle K_1, \dots, K_{r}\rangle$ of subcomplexes of $2^{[m]}$ is collectively $(r,s)$-unavoidable if for each ordered collection $(A_1, \dots, A_r)$ of disjoint sets in $[m]$ there exists a set of indices $\{i_1, \dots, i_s\}$ such that $A_{i_j} \in K_{i_j}$ for all $j=1, \dots, s$.
\end{definition}





\subsection{Brief review of Tverberg type theorems}

Radon's theorem (or Radon's lemma) is one of the most famous result in discrete geometry. It claims that for a set of $d+2$ points in $\mathbb{R}^d$ there is a partition into two subsets such that the corresponding convex hulls have a nonempty intersection. 

A generalization of this theorem was given by Helge Tverberg in \cite{T}.

\begin{theorem}[Tverberg's theorem]
  Suppose $m\geq (r-1)(d+1)+1$. Then any set $S$ of $m$ points in $\mathbb{R}^d$  can be partitioned into $r$ nonempty, pairwise disjoint subsets  $S_1, \dots, S_r$ in such a way that $\cap_{i=1}^{r}\text{conv}(S_i)\neq \emptyset.$   
\end{theorem}




The original Tverberg's theorem can be reformulated. Suppose $N=(r-1)(d+1)$, then for each affine map $f:\Delta^{N} \rightarrow \mathbb{R}^{d}$ there exist $r$ nonempty disjoint faces $\Delta_1, \dots, \Delta_r$ of $\Delta^{N}$ such that $f(\Delta_1)\cap \dots \cap f(\Delta_r) \neq \emptyset$.

This theorem may be extended to the Topological Tverberg theorem \cite{Vol, OTGTT}. Suppose $f\colon \Delta^{(r-1)(d+1)} \rightarrow \mathbb{R}^d$ is an arbitrary continuous map  and $r=p^{\nu}$ is a prime power. Then the statement of the theorem remains true: there exist $r$ nonempty disjoint faces $\Delta_1, \dots, \Delta_r$ of $\Delta^{N}$ such that $f(\Delta_1)\cap \dots \cap f(\Delta_r) \neq \emptyset$.

There are different ways to extend and generalise the Topological Tverberg theorem. For example, we may demand that for a chosen simplicial complex $K \subseteq \Delta^{N}$ simplices $\Delta_i$ lie in $K$ . In \cite{TverType} the authors proved, among other things, the following theorem:

\begin{theorem}[Tverberg theorem for collectively unavoidable complexes]
\label{TCUC}
Suppose ${N=(r-1)(d+2)}$ and $m=N+1$. Let $\mathcal{K}=\langle K_1, \dots, K_r \rangle$ be a collectively $r$-unavoidable family of subcomplexes of $\Delta^{N}$, where $r=p^{\nu}$ is a power of a prime number. Assume that $K_i$ is $(m,k)$-balanced for each $i=1, \dots, r$. Then for each continuous map $f \colon \Delta^{N} \rightarrow \mathbb{R}^d$ there exist $r$ disjoint faces $\Delta_i$ of $\Delta^N$ such that $f(\Delta_1)\cap \dots \cap f(\Delta_r) \neq \emptyset$ and $\Delta_i \in K_i$ for each $i=1, \dots, r$.  
\end{theorem}

Another direction is to color vertices and to demand faces to be rainbow. In \cite{ColTver}, in particular, the authors obtained the following theorem:

\begin{theorem}[ Balanced colored Tverberg theorem]
\label{BCT}
 Assume that $r = p^{\nu}$ is a prime power. Let integers $d \geq 1$, $k \geq 0$, and $0 < s \leq r$ be such that $r(k - 1) + s = (r - 1)d$. Let $[m] = C_{1} \uplus \dots \uplus C_{k+1}$ be a coloring of vertices of $\Delta^{N}$, where $N+1=m = (2r - 1)(k + 1)$ and $|C_i| = 2r - 1$ for each $i=1, \dots, k+1$. Then for every continuous map $f \colon \Delta^{N} \rightarrow \mathbb{R}^d$
there are $r$ pairwise disjoint rainbow faces $\Delta_1, \dots, \Delta_r$ of $\Delta^{N}$ such that $f(\Delta_{1}) \cap \dots \cap f(\Delta_{r}) \neq \emptyset $, with $dim(\Delta_i) \leq k$ for $i=1, \dots, s$ and $dim(\Delta_j) \leq k-1$ for $j=s+1, \dots, r$.
\end{theorem}

Further we shall give an extension of this theorem.

\section{Colored Tverberg theorem for rainbow-unavoidable complexes}

\begin{definition} Suppose $[m]$ is colored. A partition $(A_{1}, \dots, A_{r}, B)$ of $[m]$ is called rainbow if $A_i$ is rainbow for each $i=1, \dots,r$.
\end{definition}

\begin{definition} Let  $\mathcal{K}=\langle K_{1}, \dots, K_{r}\rangle$ be a family of simplicial complexes,  $K_i \subseteq \Delta_{m}$. Suppose $[m]$ is colored and $K_i$ is rainbow for each $i=1,\dots, r$. Then $\mathcal{K}$ is called \textit{$(r,s)$-rainbow unavoidable} if for each ordered rainbow partition $(A_{1}, \dots, A_{r}, B)$  there exists a set of indices $\{i_1, \dots, i_s\}$ such that $A_{i_j} \in K_{i_j}$ for each $j=1, \dots, s$.
  \end{definition}

\begin{theorem}[Colored Tverberg theorem for rainbow-unavoidable complexes]
\label{CTCRUC}
 Assume that $r = p^{\nu}$ is a prime power. Let integers $d \geq 1$, $k \geq 0$, and $0 < s \leq r$ be such that $r(k - 1) + s = (r - 1)d$. Let $[m] = C_{1} \uplus \dots \uplus C_{k+1}$ be a coloring of vertices of $\Delta^{N}$, where $N+1=m = (2r - 1)(k + 1)$ and $|C_i| = 2r - 1$ for each $i=1, \dots, k+1$. Let $\mathcal{K}=\langle K_1, \dots, K_r \rangle$ be a collectively $(r,s)$-rainbow unavoidable family of subcomplexes of $\Delta^{N}$. Assume that $K_i$ is $(m,k)$-rainbow balanced for each $i=1, \dots, r$. Then for each continuous map $f \colon \Delta^{N} \rightarrow \mathbb{R}^d$ there exist $r$ disjoint faces $\Delta_i$ of $\Delta^N$ such that $f(\Delta_1)\cap \dots \cap f(\Delta_r) \neq \emptyset$ and $\Delta_i \in K_i$ for each $i=1, \dots, r$.  
\end{theorem}

\begin{proof}

This is a standard argument found, for example, in \cite{TverType, ColTver, UBUT}.

Assume the converse. Then there is a map $f \colon \Delta_{[m]} \rightarrow \mathbb{R}^{d}$ such that $${f(A_1)\cap f(A_2) \cap \dots \cap f(A_r) = \emptyset}$$ for every $(A_1, \dots, A_r, B) \in SymmDelJoin(\mathcal{K})$. We consider the
$r$-fold join $f^{*r}=\Phi_{f}$. This is $S_r$-equivariant mapping $$\Phi_{f}\colon SymmDelJoin(\mathcal{K})\rightarrow (\mathbb{R}^{d})^{*r} \cong \mathbb{R}^{dr+r-1}.$$
Let $D \cong \mathbb{R}^{d}$ be the diagonal subspace of $(\mathbb{R}^{d})^{*r}$. From the condition on $f$ it follows that the image under $\Phi_{f}$ does not intersects with the diagonal $D$. And $(\mathbb{R}^{d})^{*r} \setminus D$ is $S_r$-homotopy equivalent to a sphere $\mathbb{S}^{(d+1)(r-1)-1}$. But since $SymmDelJoin(\mathcal{K})$ is $(rk + s - 2)=(rd+r-d-2)$-connected, then we have a contradiction with Volovikov's theorem  \cite{Z21}.
\end{proof}

All we need now is to show that $SymmDelJoin(\mathcal{K})$ is $(rk + s - 2)$-connected.

\begin{theorem}
\label{HCon}
Let $[m] = C_{1} \uplus \dots \uplus C_{k+1}$ be a coloring of vertices of $\Delta_{m}$, where ${m = (2r - 1)(k + 1)}$ and $|C_i| = 2r - 1$ for each $i=1, \dots, k+1$. Let $\mathcal{K}=\langle K_1, \dots, K_r \rangle$ be a collectively $(r,s)$-rainbow unavoidable family of subcomplexes of $\Delta^{N}$. Assume that $K_i$ is $(m,k)$-rainbow balanced for each $i=1, \dots, r$. Then $SymmDelJoin(\mathcal{K})$ is $(rk + s - 2)$-connected.
\end{theorem}

\textbf{Proof of Theorem \ref{HCon} .} We shall proof this theorem with the methods of discrete Morse theory \cite{F}.  But first let us briefly discus the main definitions and ideas from this theory.

    Let $K$ be a simplicial complex. A simplex $\alpha \in K$ is denoted as $\alpha^{p}$ if $dim(\alpha)=p$. A matching $(\alpha^{p}, \beta^{p+1})$ of simplices is called a \textit{discrete vector field} if the following conditions holds:
    \begin{itemize}
        \item Each simplex of the complex appears in no more then one pair.
        \item For each mathed pair $(\alpha^{p}, \beta^{p+1})$, the simplex $\alpha^{p}$ is the facet of the simplex $\beta^{p+1}$.
    \end{itemize}
    A \textit{gradient path} is a sequence of simplices $\alpha^{p}_0, \beta^{p+1}_0, \dots, \alpha^{p}_k, \beta^{p+1}_k, \alpha^{p}_{k+1}$ satisfying the following conditions:
    \begin{itemize}
        \item  $\alpha^{p}_i$ and $\beta^{p+1}_i$ are matched for each $i=0, \dots, k$.
        \item  $\alpha^{p}_{i+1}$ is a facet of $\beta^{p+1}_i$ for each $i=0, \dots, k$.
        \item $\alpha^{p}_i \neq \alpha^{p}_i$.
    \end{itemize} 
    A gradient path is closed if $\alpha^p_0=\alpha^p_{k+1}$. A \textit{discrete Morse function} on a simplicial complex is a discrete vector field without closed paths. For a fixed discrete Morse function a simplex $\sigma \in K$ is called \textit{critical} if $\sigma$ is unmatched.

    The main idea of this theory is to contract all matched pairs of simplices along the gradient paths. During this procedure we reduce simplicial complex $K$ to a cell complex with critical simplices as cells. This leads us to the following theorem: 
    \begin{theorem}[\cite{F}]
        \label{MTDTM}
        Assume that we have a fixed discrete Morse function on $K$. Suppose that there is only one critical simplex of dimension $0$ and all other critical simplices have dimension $\geq N$. Then $K$ is $(N-1)$-connected.
    \end{theorem}

    In order to proof Theorem \ref{HCon} we need to construct a discrete Morse function on $\mathcal{K}$ such that there is only one critical simplex of dimension $0$ and all other critical simplices have dimension $\geq rk + s - 1 $.

 
 Note that every simplicial complex $K_i$ consists of all rainbow simplices of dimension $k-1$ plus a family of rainbow simplices of dimension $k$. Note that if a rainbow simplex has dimension $k$ then this simplex is maximal and contains all possible $k+1$ colors.
 
  A simplex $(A_1, \dots A_r, B)$ is called maximal if it is not a face of any other simplex of $SymmDelJoin(\mathcal{K})$.

 A partition $(A_1, \dots A_r, B)$ is called admissible if there is a permutation $\pi$ such that $A_{\pi(i)} \in K_{i}$ (and hence $A_i$ is rainbow) for each $i=1, \dots, r$. In other words, a partition is admissible if it belongs to $SymmDelJoin(\mathcal{K})$.

Now we shall start the construction the desired Morse function. The construction is done in $r$ big steps. Every big step is divided into $k+1$ small steps. Every step is a matching procedure of simplices from $SymmDelJoin(\mathcal{K})$.

\bigskip

\textbf{Step 1.1} Every color contains $2r-1$ vertices. Suppose that they are enumerated by $\{1, \dots, 2r-1\}$. Assume we have a simplex  $(A_1, \dots A_r, B)$. Set $$a_{1}^{1}=min\left(\left(A_1\cup B\right) \cap C_1\right).$$ Then we match $(A_1, \dots A_r, B \cup a_{1}^{1})$ with $(A_1 \cup a_{1}^{1}, \dots A_r, B)$ whenever both these simplices are admissible. Now an admissible simplex of the type $(A_1 \cup a_{1}^{1}, \dots A_r, B)$ is unmatched iff $A_1 \cup a_{1}^{1}=a_{1}^{1}$ and   $B=[m]\setminus \{a^1_1\}$. This is a simplex of dimension $0$ and it will stay unmatched until the end of the procedure. If a simplex of the type $(A_1, \dots A_r, B \cup a_{1}^{1})$ is unmatched, then either $A_1$ contains a vertex colored by $C_1$, or by adding $a^1_1$ to $A_1$ we make $(A_1, \dots A_r, B \cup a_{1}^{1})$ not admissible.

\bigskip

\textbf{Step 1.2}
Set $$a_{1}^{2}=min\left(\left(A_1\cup B\right) \cap C_2\right).$$ Then we match $(A_1, \dots A_r, B \cup a_{1}^{2})$ with $(A_1 \cup a_{1}^{2}, \dots A_r, B)$ whenever both these simplices are unmatched and admissible.

\begin{itemize}
   \item If a simplex of the type $(A_1, \dots A_r, B \cup a_{1}^{2})$ (with dimension $>0$) is unmatched, then either $A_1$ contains a vertex colored by $C_2$, or by adding $a^2_1$ to $A_1$ we make ${(A_1, \dots A_r, B \cup a_{1}^{1})}$ not admissible. This is \textit{Step 1.2-Type 1 unmatched simplex}.
   \item If a simplex of the type $(A_1 \cup a_{1}^{2}, \dots A_r, B)$ is unmatched, then $(A_1, \dots A_r, B \cup a_{1}^{2})$ have been matched by $a_{1}^{1}$. Therefore $a_{1}^{1} \in B$, $A_1 \cap C_1 = \emptyset$, and $|A_1 \cup a_{1}^{2}|=k$. This is \textit{Step 1.2-Type 2 unmatched simplex}.
\end{itemize}

\textbf{Steps $1.3, \dots 1.k+1$} are done in the similar way. Further we shall use the similar notation. ``Step i.j -- Type 1'' means we can not move $\{a_i^j\}$ from $B$ to $A_i$. ``Step i.j -- Type 2'' means we can not move $\{a_i^j\}$ from $A_i$ to $B$.  The following lemmas are needed for the sequel.

\begin{definition}
    A simplex $(A_1, \dots, A_r, B)$ is called ``step $i$''-maximal admissible if $|A_i|=k$ and by adding an arbitrary vertex to $A_i$ we obtain a not admissible simplex. 
\end{definition}

\begin{lemma}
    Excluding the unique unmatched zero-dimensional simplex, if a simplex  $(A_1, \dots A_r, B)$ is unmatched after the first step, then:
\begin{itemize}
    \item either $|A_1|=k+1$, or 
    \item $|A_1|=k$ and $(A_1, \dots A_r, B)$ is ``step 1''-maximal admissible.
\end{itemize}    
\end{lemma}

\textbf{Step 2.} During this step we will treat $A_2$.

\bigskip

\textbf{Step 2.1.} Set $$a_{2}^{1}=min\left(\left(\left(A_2\cup B\right) \setminus \left[1,a_{1}^{1}\right]\right)\cap C_{1}\right).$$ Then we match  $(A_1, A_2, \dots A_r, B \cup a_{2}^{1})$ with $(A_1, A_2 \cup a_{2}^{1}, \dots A_r, B)$ whenever both these simplices are admissible and unmatched. 
\begin{itemize}
    \item If a simplex of the type $(A_1, A_2, \dots A_r, B \cup a_{2}^{1})$ (with dimension $>0$) is unmatched, then either $A_2$ contains a vertex colored by $C_1$, or by adding $a^1_2$ to $A_2$ we make ${(A_1, \dots A_r, B \cup a_{1}^{1})}$ not admissible. This is \textit{Step 2.1-Type 1 unmatched simplex}.
    \item If a simplex of the type $(A_1, A_2 \cup a_{2}^{1}, \dots A_r, B)$ is unmatched, then it is ``step 1''-maximal admissible and $|A_2 \cup a_{2}^{1}|=k+1$. This is \textit{Step 2.1-Type 2 unmatched simplex}.
\end{itemize}

\textbf{Step 2.2.} Set $$a_{2}^{2}=min\left(\left(\left(A_2\cup B\right) \setminus \left[1,a_{1}^{2}\right]\right)\cap C_{2}\right).$$ Then we match  $(A_1, A_2, \dots A_r, B \cup a_{2}^{2})$ with $(A_1, A_2 \cup a_{2}^{2}, \dots A_r, B)$ whenever both these simplices are admissible and unmatched. 

\textbf{Steps $2.3, \dots, 2.k+1$} are done in the similar way.

\begin{lemma}
    Excluding the unique unmatched zero-dimensional simplex, if a simplex  $(A_1, \dots A_r, B)$ is unmatched after step $2$, then:
\begin{itemize}
    \item either $|A_2|=k+1$, or 
    \item $|A_2|=k$ and $(A_1, \dots A_r, B)$ is ``step 2''-maximal admissible.
\end{itemize} 
\end{lemma}

\textbf{Steps 3,4, ... , and r - 1} go analogously. Now we can formulate the following lemma.

\begin{lemma} For $j=1, \dots, r-1$ the numbers $a_{j}^{i}$ are well-defined. Note that after the step $r-1$ if $(A_1, \dots A_r, B)$ is unmatched, then for  $i=1,\dots, r$ we have:
\begin{itemize}
    \item either $|A_i|=k+1$, or 
    \item $|A_i|=k$ and $(A_1, \dots A_r, B)$ is ``step $i$''-maximal admissible.
\end{itemize}
\end{lemma}

\textbf{Step r}

\bigskip

\textbf{Step r.1.} Suppose $(A_1, \dots, A_r, B)$ is unmatched. Set $$a_{r}^{1}=min\left(\left(\left(A_{r}\cup B\right) \setminus \left[1,a_{r-1}^{1}\right]\right) \cap C_{1}\right).$$ It is possible that $(A_{r}\cup B \setminus \left[1,a_{r-1}^{1}\right]) \cap C_{2}=\emptyset$ and $a_{r}^{1}$ is ill-defined. Then we see that $A_i \cap C_1>0$ for every $i=1, \dots, r$. Such unmatched simplices are called \textit{Step r.1 -- Type 3}. If $a_{r}^{1}$ is well defined, then we proceed as before.

\bigskip

\textbf{Step r.2.} Set $$a_{r}^{2}=min\left(\left(\left(A_{r}\cup B \right) \setminus \left[1,a_{r-1}^{2}\right]\right)\cap C_{2}\right).$$ If $a_{r}^{2}$ is ill-defined, then $A_i \cap C_2 >0$ for $i=1, \dots, r$ and this is \textit{Step r.2 -- Type 3} unmatched simplex. Otherwise we proceed as before.

\bigskip

\textbf{Steps r.2 -- r.k+1} are done in the same way.



\begin{lemma}
   If a simplex $(A_1, \dots A_r, B)$ of dimension $>0$ and is unmatched, then this simplex is maximal in $SymmDelJoin(\mathcal{K})$.
\end{lemma}
\begin{proof}

If a simplex $(A_1, \dots A_r, B)$ of dimension $>0$ and is unmatched, then we have $|A_r| \geq k$. Indeed, if $(A_1, \dots A_r, B)$ is always step $r$ -- type $3$ or type $1$, then all colors are present in $A_r$ and $|A_r|=k+1$. If for some $j$ the simplex $(A_1, \dots A_r, B)$ is step $r.j$ -- type $2$, then it is ``step $r$''-maximal admissible and $|A_r|=k$. Then for each $i=1, \dots, r$ we have $|A_i|\geq k$. Also, since for each $i=1, \dots, r$ either $|A_i|=k+1$, or $|A_i|=k$ and $(A_1, \dots A_r, B)$ is ``step $i$''-maximal admissible, we can see that $(A_1, \dots A_r, B)$ is not a face of any other simplex from $SymmDelJoin(\mathcal{K})$, hence this simplex is maximal.  So we can see that if a simplex $(A_1, \dots A_r, B)$ is unmatched after the procedure, then it is maximal in $SymmDelJoin(\mathcal{K})$.
\end{proof}

\begin{statement}

Excluding the unique unmatched zero-dimensional simplex, all unmatched simplices have at least $rk+s$ vertices.
\end{statement}

\begin{proof}
    Let $(A_1, \dots A_r, B)$ be an unmatched simplex. Note that if $|A_j|=k$ for some $j$, then $a_{j}^{i_{j}}$ is well defined for the missing color $C_{i_j}$. Suppose $\{A_{j_1}, \dots, A_{j_l}\}$ is the set of all $(k-1)$-dimensional simplices. Without loss of generality, we can assume that $A_i \in K_i$ for $i=1, \dots, r$. Also, assume $A_j'=A_j$ if $|A_j|=k+1$ and $A_j'=A_j \cup \{a_{j}^{i_{j}}\}$ if $|A_j|=k$. Consider the partition $(A_1', \dots, A_r', B')$. Since the collection $\langle K_1, \dots, K_r \rangle$ is $(r,s)$-rainbow unavoidable, there are at least $s$ simplices such that $A_{i_j}' \in K_{i_j}$ for $j=1, \dots, s$. Note that since $A_j \cup \{a_{j}^{i_{j}}\} \notin K_{j}$, we have at least $s$ simplices from $(A_1, \dots A_r, B)$ of dimension $k$, then $(A_1, \dots A_r, B)$ is at least $(rk+s-1)$-dimensional. 
\end{proof}

So from Theorem \ref{MTDTM} follows that $SymmDelJoin(\mathcal{K})$ is $rk+s-2$-connceted. Now we must only prove that the matching is without closed paths 


Suppose we have a gradient path $$\alpha_0^p, \beta_0^{p+1}, \alpha_1^p,..., \beta_m^{p+1}, \alpha_{m+1}^{p}.$$ For any simplex $\alpha$ we can define a sequence $$\Pi(\alpha)\colon=(a^1_1, a^2_1, \dots, a^1_r \dots, a^{k+1}_r),$$ where $a^i_j$ are from the matching procedure. They are well-defined for $j<r$. If $\alpha$ is matched on a step $i$, then $a^i_j$ is well-defined. If $a^i_r$ is ill-defined, then we can assume that $a^i_r=\infty$.

\begin{lemma}

 The gradient path $\Pi(\alpha)$ is strictly decreasing with respect to lexicographic order. From this follows that the gradient path is not closed.
\end{lemma}

\begin{proof}
    It suffice to proof the lemma for a short path $\alpha_0^p, \beta_0^{p+1}, \alpha_1^p, \beta_1^{p+1}.$ Let us consider 4 cases.
    \begin{enumerate}
        \item Suppose the pair $\alpha_0^p, \beta_0^{p+1}$ is matched by adding color $i$ to $A_j$ and $\alpha_1^p$ comes from $\beta_0^{p+1}$ by removing color $i'>i$ from $A_j$.
         \begin{itemize}
        \item either $\alpha_1^p$ is matched on Step $j.i$ by removing color $i$ from $A_j$ and then the path terminates here, or
        \item or  $\alpha_1^{p}$ is matched  before Step $j.i$.
    \end{itemize}
        \item Suppose the pair $\alpha_0^p, \beta_0^{p+1}$ is matched by adding color $i$ to $A_j$ and $\alpha_1^p$ comes from $\beta_0^{p+1}$ by removing color $i'<i$ from $A_{j}$. Then $\alpha_1^p$ is matched before Step $j.i$.
        \item Suppose the pair $\alpha_0^p, \beta_0^{p+1}$ is matched by adding color $i$ to $A_j$ and $\alpha_1^p$ comes from $\beta_0^{p+1}$ by removing color $i'$ from $A_{j'}$ for $j'<j$.
     \begin{itemize}
        \item either  $\alpha_1^{p}$ is matched by adding color $i'$ on Step $j',i'$, or
        \item or $\alpha_1^{p}$ is matched before Step $j',i'$.
    \end{itemize}
    \item Suppose the pair $\alpha_0^p, \beta_0^{p+1}$ is matched by adding color $i$ to $A_j$ and $\alpha_1^p$ comes from $\beta_0^{p+1}$ by removing color $i'$ from $A_{j'}$ for $j'>j$.
         \begin{itemize}
        \item either  $\alpha_1^{p}$ is matched by removing color $i$ from $A_j$ and then the path terminates here, or
        \item or  $\alpha_1^{p}$ is matched before Step $j'.i'$.
    \end{itemize}
    \end{enumerate}

So the last lemma is proved. This completes the proof of Theorem \ref{CTCRUC}.

\end{proof}

\section{"Monochromatic" Tverberg theorem for collectively $(r,s)$-unavoidable complexes}

We can slightly extend Theorem \ref{TCUC} by demanding faces to be $(r,s)$-unavoidable.

\begin{theorem}
\label{TTRSU}
 Let $\mathcal{K}=\langle K_1, \dots, K_r \rangle$ be a collectively $(r,s)$-unavoidable family of subcomplexes of $\Delta^{N}$, where $r=p^{\nu}$ is a power of a prime number. Suppose $N=(r-1)(d+2)-s+1$ and $m=N+1$. Assume that $K_i$ is $(m,k)$-balanced  for each $i=1, \dots, r$. Then for each continuous map $f \colon \Delta^{N} \rightarrow \mathbb{R}^d$ there are $r$ pairwise disjoint faces $\Delta_1, \dots, \Delta_r$ such that $f(\Delta_1)\cap \dots \cap f(\Delta_r) \neq \emptyset$ and $\Delta_i \in K_i$ for each $i=1, \dots, r$. 
\end{theorem}

Since $\mathcal{K}$ is $(r,s)$-unavoidable, we obtain that the configuration space $SymmDelJoin(\mathcal{K})$ is $(m-r+s-2)$-connected. Then $m-r+s-2 \geq (r-1)(d+1)-1$ is equivalent to $N \geq (r-1)(d+2)-s+1$. Assume that $s=1$, then we are in the conditions of Theorem \ref{TCUC}. And if $s=r$, then from $(r,s)$-unavoidability it follows that $K_i$ is exactly $\Delta_{[m]}$ and $N=(r-1)(d+2)-r+1=(r-1)(d+1)$, so in this case we are in the conditions of the Topological Tverberg theorem. Note that the parameter $k$ is not fixed and may grow with the growth of $s$.

\section{Collective $(r,s)$-unavoidability in graph terms}

Let $\mathcal{K}=\langle K_{1}, \dots, K_{r} \rangle$ be a family of simplicial complexes. Suppose $K_i$ is $(m,k)$-balanced (or rainbow balanced) for every $i=1, \dots, r$. By $\mathcal{A}^{i}$ denote a relative complement $\Delta^{(k)}_{[m]} \setminus K_i$ (or $Col\left(\Delta^{(k)}_{[m]}\right) \setminus K_i$). $\mathcal{A}^i$ is a family of sets $\{A^i_1, \dots, A^i_{l_i}\}$. Since $K_i$ is  $(m,k)$-balanced (or rainbow balanced) we see that $|A^i_j|=k+1$ for each $j=1, \dots, l_i$. Then for $\mathcal{K}$ we assign a graph $\Gamma=\Gamma(\mathcal{K})=(V(\mathcal{K}), E(\mathcal{K}))$. The vertex set $V(\mathcal{K})$ is a set of pairs $(i,j)$, where $i=1, \dots, r$ and for each $i$ we have $j=1, \dots, l_i$. In other words, for every set $A^i_j$ we assign a vertex. Two vertices $(i,j)$ and $(i',j')$ are connected with an edge whenever $i \neq i'$ and $A^i_j \cap A^{i'}_{j'}=\emptyset$. So $\Gamma(\mathcal{K})$ is $r$-partite Kneser graph $KG(\mathcal{A})$.

\begin{proposition}
\label{PROPCLIQ}
If $\mathcal{K}$ is $(r,s)$-unavoidable, then there is no $(r-s+1)$-vertex clique in $\Gamma(\mathcal{K})$. 
\end{proposition}

\begin{proof}
Assume the converse. Then there is a collection of $r-s+1$ pairwise disjoint sets $\{A^{i_1}_{j_1}, \dots, A^{i_{r-s+1}}_{j_{r-s+1}}\}$. Since $A^{i_l}_{j_l} \notin K_{i_l}$ for each $l=1, \dots, r-s+1$ and $\mathcal{K}$ is $(r,s)$-unavoidable, we have a contradiction.

\end{proof}

\begin{remark}
The converse is not true in general. Suppose $m>(k+2)(r-s+1)$ and $K_i=\Delta^{(k)}_{[m]}$ for $i=1, \dots, r$.  Then $\mathcal{K}=\langle K_{1}, \dots, K_{r} \rangle$ is $(m,k)$-balanced. Note that in $\Gamma(\mathcal{K})$ there are no vertices at all, and hence no $(r-s+1)$-vertex clique. But we have a partition $(A_1, \dots, A_r, B)$, where $|A_i|=k+2$ for $i=1, \dots, r-s+1$ and $A_i=\emptyset$ for $i=r-s+2, \dots, r$. Since $A_i \notin K_i$ for $i=1, \dots, r-s+1$, we see that $\mathcal{K}$ is not $(r,s)$-unavoidable.
\end{remark}

\begin{remark}
    The converse is true in the settings of Theorem \ref{CTCRUC}. Indeed, suppose $K_i$ is $(m,k)$-rainbow balanced for $i=1, \dots, r$ and $\Gamma(\mathcal{K})$ hase no $(r-s+1)$-vertex clique. Let $(A_1, \dots, A_r, B)$ be a rainbow partition.  Since all sets are rainbow we see that $|A_i|\leq k+1$ for $i=1, \dots, r$. Therefore if $A_i \notin K_i$, then $A_i \in \mathcal{A}^{i}$. Hence if there is no $(r-s+1)$-vertex clique, then $\mathcal{K}$ is $(r,s)$-rainbow unavoidable.
\end{remark}

\begin{remark}
    In Theorem \ref{TTRSU} we can replace the condition ``$\mathcal{K}$ is $(r,s)$-unavoidable'' with the condition ``the graph $\Gamma(\mathcal{K})$ does not have a clique on $r-s+1$ vertices''. The proof remains the same.
    
\end{remark}

\section{Comparison with previous results}

Theorem \ref{BCT} follows from Theorem \ref{CTCRUC} as a special case. Indeed, suppose $K_i=Col\Delta_{[m]}^{(k)}$ for $i=1, \dots, s$ and $K_{i}=Col\Delta_{[m]}^{(k-1)}$ for $i=s+1, \dots, r$. Since $\Gamma(\mathcal{K})$ does not have a clique on $r-s+1$ vertices, we see that this family of simplices is $(r,s)$-rainbow unavoidable. 

\begin{acknowledgment} I wish to thank Gaiane Panina for suggested problem and advice.
\end{acknowledgment}


\end{document}